\documentclass[reqno]{amsart}
\usepackage[english]{babel}
\usepackage{amsmath}
\usepackage{amsfonts}
\usepackage{mathrsfs}
\usepackage{amstext}
\usepackage{amssymb}
\usepackage{amsthm} 
\usepackage{hyperref} 
\usepackage{esint} 
\usepackage{enumerate}

\usepackage{color}
\usepackage{scrtime}
\usepackage{float}
\usepackage{microtype} 
\usepackage{graphicx}
\usepackage{tikz}
\usetikzlibrary{calc,fadings,decorations.pathreplacing,graphs}
\usepackage{hyperref}
\usepackage{pgfplots}
\usepackage{mathtools}
\pgfplotsset{compat = newest}

\pagecolor{white}

\newtheorem{thm}{Theorem}
\newtheorem{prop}[thm]{Proposition}
\newtheorem{cor}[thm]{Corollary}

\newtheorem{lem}[thm]{Lemma}

\begin{document}

\title[Continuous harmonic functions not in $H^s$]{Continuous harmonic functions on a ball that are not in $H^s$ for $s>1/2$}

\author{Roberto Bramati}
\address{Department of Mathematics: Analysis, Logic and Discrete Mathematics, Ghent University, Krijgslaan 281, 9000 Ghent, Belgium}
\email{Roberto.Bramati@UGent.be}

\author{Matteo Dalla Riva}
\address{Dipartimento di Ingegneria, Universit\`a degli Studi di Palermo, Viale delle Scienze, Ed. 8, 90128 Palermo, Italy}
\email{matteo.dallariva@unipa.it}

\author{Brian Luczak}
\address{Department of Mathematics, Vanderbilt University, 1326 Stevenson Center Lane, Nashville, TN 37212, USA}
\email{brian.b.luczak@vanderbilt.edu}

\begin{abstract}
We show that there are harmonic functions on a  ball ${\mathbb{B}_n}$ of $\mathbb{R}^n$, $n\ge 2$,  that are continuous, and even H\"older continuous, up to the boundary but not  in the Sobolev space $H^s(\mathbb{B}_n)$ for   $s$  bigger than a certain sharp bound. The idea for the construction is inspired by the two-dimensional example of a harmonic continuous function with infinite energy presented by Hadamard in 1906. To obtain examples in any dimension $n\ge 2$ we exploit certain series of spherical harmonics. 
\end{abstract}




\maketitle

\section{Introduction} 


To the best of the authors' knowledge, there are two classic examples of ill-behaved continuous harmonic functions of the kind we consider in this paper. The first one is due to Prym \cite{Pr71} and was published in 1871 and the second (and more famous) one appeared in 1906  in a paper   by Hadamard \cite{Ha06}. Both authors sought to understand the potential limitations of the Dirichlet principle, that is, of the  equivalence between the Dirichlet boundary value problem and its variational formulation (see for example  Maz'ya and Shaposnikova \cite{MaSh98} and Bottazzini and Gray \cite{BoGr13} for a historical insight). Prym and Hadamard considered a two-dimensional Dirichlet problem in a ball and showed examples of harmonic functions that are continuous up to the boundary and have infinite energy. In  modern terms, these functions  do not belong to the space $H^1$. 

We might say that the present paper stems from that of Hadamard. So we will start by summarizing Hadamard's idea.  Let $\mathbb{B}_2$ be the unit ball  in $\mathbb{R}^2$ with center at the origin. Let $u$ be a harmonic function on $\mathbb{B}_2$   that can be written in the form
\begin{equation}\label{hadamard.1}
u(x)=\sum_{k=1}^\infty a_k r^k \cos(k t)\ \text{for }x\in\mathbb{B}_2\setminus\{0\},\qquad u(0)=0
\end{equation}
where $r=|x|$ and $t=\arg(x/|x|)$ and where $\{a_k\}_k$ is a suitable sequence of real numbers. Then Hadamard notices that  $u$ has energy 
\[
\int_{\mathbb{B}_2}|\nabla u|^2\,dx=\pi\sum_{k=1}^\infty ka_k^2
\]
and for 
\begin{equation}\label{hadamard.2}
a_k:=
\begin{cases}
0&\text{if }k\neq 2^{2j}\text{ for all }j\in\mathbb{N}\,,\\
2^{-j}&\text{if }k= 2^{2j}\text{ for some }j\in\mathbb{N}\,\\
\end{cases}
\end{equation}
the series that defines $u$ converges uniformly on the closure $\overline{\mathbb{B}_2}$ of $\mathbb{B}_2$, whereas the series that gives the energy of $u$ diverges.  It follows that $u$ is continuous on  $\overline{\mathbb{B}_2}$, harmonic in $\mathbb{B}_2$,  and does not belong to $H^1({\mathbb{B}_2})$.

With a modification of Hadamard's approach we may also obtain regularities below $H^1$. To do so,  we  replace the argument based on the computation of the   energy with an argument based on the trace theorem and on a suitable characterization of fractional order Sobolev spaces on the sphere   $\mathbb{S}_{1}$. We shall see that, with a suitable choice of the sequence $\{a_k\}_k$, we can make the function in \eqref{hadamard.1}  not belong to $H^{\sigma+1/2}(\mathbb{B}_2)$ for any $\sigma>0$ (see Proposition \ref{withanyYk} and the remark after it).

Moreover, using some known results on Weierstrass’ functions we can keep the H\"older continuity in check. Namely, for a given $0\le\alpha<1$ we can find a sequence $\{a_k\}_k$ such that the  harmonic function $u$ defined by \eqref{hadamard.1} has the following properties:
\begin{enumerate}[(i)]
\item\label{p1n=2} $u$ belongs to $C^{0,\alpha}(\overline{\mathbb{B}_2})$ but not to $C^{0,\alpha+\epsilon}(\overline{\mathbb{B}_2})$ for any $\epsilon>0$,
\item\label{p2n=2} $u$ does not belong to $H^{\sigma+1/2}(\mathbb{B}_2)$ for any $\sigma>0$ if $\alpha=0$ and  $\sigma\ge \alpha$ if $\alpha>0$
\end{enumerate}
(cf.~Propositions \ref{notCbeta>0} and \ref{notCbeta>alpha}, see also \cite{Lu19}).
As usual, we say that a function is of class $C^{0,\alpha}$ if it is H\"older continuous with exponent $\alpha$ and we understand that it is $C^{0,0}$ if it is continuous.

It is possible to see that the bounds for $\sigma$ in \eqref{p2n=2} are sharp. Indeed, if \eqref{p1n=2} is verified, then it follows that $u$ belongs to $H^{1/2}(\mathbb{B}_2)$ for $\alpha=0$ and to any $H^{\sigma+1/2}(\mathbb{B}_2)$ with $0<\sigma<\alpha$ if $\alpha>0$ (cf.~Proposition \ref{prop:sharp}). That is to say, once \eqref{p1n=2} is verified, the bounds for $\sigma$ in \eqref{p2n=2} ($\sigma>0$ for $\alpha=0$ and $\sigma\ge \alpha$ if $\alpha>0$) cannot be lowered.

Our  goal in the paper is to obtain a similar result for any dimension $n\ge 2$. Namely, we want to construct a harmonic function $u$ on the ($n\ge2$)--dimensional ball $\mathbb{B}_n$ that satisfies the $n$-dimensional analog of properties \eqref{p1n=2} and \eqref{p2n=2}: 
\begin{enumerate}[(I)]
\item\label{p1} $u$ belongs to $C^{0,\alpha}(\overline{\mathbb{B}_n})$ but does not belong to $C^{0,\alpha+\epsilon}(\overline{\mathbb{B}_n})$ for any $\epsilon>0$,
\item\label{p2} $u$ does not belong to $H^{\sigma+1/2}(\mathbb{B}_n)$ for any $\sigma>0$ if $\alpha=0$ and $\sigma\ge \alpha$ if $\alpha>0$.
\end{enumerate}
As for the two-dimensional case, if $u$ satisfies \eqref{p1} then it belongs to $H^{1/2}(\mathbb{B}_n)$ if $\alpha=0$ and to $H^{\sigma+1/2}(\mathbb{B}_n)$ if $0<\sigma<\alpha$ and, in this sense, the bounds for $\sigma$ are sharp (see Proposition \ref{prop:sharp}).

Our first attempt  consists in replacing the cosine functions in \eqref{hadamard.1} with a sequence of spherical harmonics $\{Y_k\}_k$ with each $Y_k$ of degree $k$. Namely, we take
\begin{equation}\label{idea0}
u(x):=\sum_{k=0}^\infty{a_k}r^k\, Y_k(\theta)\quad\text{for }x\in \overline{\mathbb{B}_n}\setminus\{0\}\qquad\text{and }u(0):=a_0Y_0\,,
\end{equation}
where $r:=|x|$ and $\theta:=x/|x|$. For a suitable choice of the sequence $\{a_k\}_k$ the function in \eqref{idea0} is harmonic and continuous in $\overline{\mathbb{B}_n}$.  We will see, however, that if we take a random sequence of spherical harmonic $\{Y_k\}_k$, then $u$ does not stay in  $H^{\sigma+1/2}(\mathbb{B}_n)$ only for $\sigma>(n-2)/2$, which, for $n\ge 3$, is  far from the optimal bound $\sigma>0$ (cf.~Proposition \ref{withanyYk}). In addition, it is not clear how we can take control of the H\"older continuity.  A better result can be obtained using not just any sequence of spherical harmonics, but a specific one,  the so called ``highest-weight spherical harmonics" (see Sogge \cite{So15}). With the right choice of the coefficients $\{a_k\}_k$ the corresponding function $u$ will be harmonic on $\mathbb{B}_n$ and will satisfy \eqref{p1}. But still, we will not get the sharp bounds in \eqref{p2}. We will instead have that 
\begin{itemize}
\item[(II')]\label{p2'} $u$ does not belong to $H^{\sigma+1/2}(\mathbb{B}_n)$ for any $\sigma>(n-2)/4$ if $\alpha=0$ and any $\sigma\ge\alpha+(n-2)/4$ if $\alpha>0$
\end{itemize}
(cf.~Propositions \ref{notCbeta>0} and \ref{notCbeta>alpha}).
The bound $(n-2)/4$ is smaller, and thus better, than the  bound $(n-2)/2$ that we get for a random sequence $\{Y_k\}_k$; however, we still do not have the sharp bounds. In addition, the bounds in (II') get worse as $n$ increases. For $n\ge 4$ we have functions that are actually in $H^1(\mathbb{B}_n)$, and so in this case we don't even get to the Hadamard's result of a function with infinite energy. 

The dependence on $n$ of  the bounds in (II') is related to the asymptotic behavior of the ratio between the sup-norm and the two-norm of the elements of the sequence $\{Y_k\}_k$.  For $n\ge 3$ this ratio is expected to grow with $k$. The slower this growth, the smaller are the bounds for $\sigma$ and the better is our result. The worse case is that of zonal spherical harmonics, whose ratio goes like $k^{\frac{n-2}{2}}$ and yields the $\sigma>(n-2)/2$ bound. For the highest-weight spherical harmonics the situation is a little better, the ratio grows like $k^{\frac{n-2}{4}}$ and we obtain (II'), but not yet the sharp bounds for $\sigma$.

So, although it looked like the very natural way to extend Hadamard's result, it seems that using a function as in \eqref{idea0} with an explicit sequence of harmonics $\{Y_k\}_k$ will hardly give us the sharp bounds for $\sigma$. In our paper we show two different ways to obtain these bounds. In the first example the sequence $\{Y_k\}_k$ is not explicit; in the second one the construction of $u$ is somehow less direct.
\begin{enumerate}[(1)]
\item As mentioned above, there is a relation between the growth rate of the sup-norm/two-norm ratio of $\{Y_k\}_k$ and the bounds for $\sigma$. For the sequence of highest-weight spherical harmonics the sup-norm/two-norm ratio grows slower than that of a random sequence, and this leads to smaller bounds. For a generic sequence it is even better though. Indeed, Neuheisel \cite{Ne00} proved that the ratio of ``almost all" sequences grows like $\sqrt{\ln k}$. This means that there are plenty of sequences $\{Y_k\}_k$ with ratio that grows like $\sqrt{\ln k}$ and if we plug one of them in the expression \eqref{idea0}, we obtain a continuous harmonic function $u$ that, for $\alpha=0$,  satisfies \eqref{p1} and \eqref{p2} (cf.~Theorem \ref{Neu}). So here it is, at least for $\alpha=0$ we have an ill-behaved harmonic function with the sharp bound on $\sigma$.  On the downside, we don't know any explicit sequence of harmonics with the logarithmic ratio growth, and thus we don't have any explicit expression for $u$. Also, we don't know if we can extend this approach to the case where $\alpha>0$.   
\item For the second example we introduce a specific domain $\Omega$  and, at first, we  construct ill-behaved harmonic functions on $\Omega$ rather than $\mathbb{B}_n$. The idea for the definition of  $\Omega$ is taken from a paper by Costabel \cite{Co19} and the shape of $\Omega$ is such that its boundary contains a cylinder-like portion $\mathbb{S}_1\times(-\delta,\delta)^{n-2}$. Then, in $\Omega$ we define a harmonic function $u$ using an analog of \eqref{idea0} with the sequence of harmonic polynomials associated to the highest-weight spherical harmonics. We can show that such function $u$ is harmonic on $\Omega$ and satisfies--for any $0\le \alpha<1$--conditions similar to \eqref{p1} and \eqref{p2}, where we just need to replace $\mathbb{B}_n$ with  $\Omega$ (cf.~Theorem \ref{inOmega}). So, on $\Omega$ we have the optimal result with the sharp bounds and for any $0\le\alpha<1$. The reason why we could have this on $\Omega$ with a function that highly resembles the one that yielded (II') on $\mathbb{B}_n$ is due to the specific shape of $\Omega$, which, in a sense, allows the singularity to ``spread" all over the cylindrical part. At this point we may bring back our result  from $\Omega$ to a ball $\mathbb{B}_n$. The function thus obtained is not explicit, but its restriction to the boundary of $\mathbb{B}_n$ is (cf.~Theorem \ref{fromOmega}). In principle, we  may obtain the function as the solution of a Dirichlet boundary  value problem or, equivalently, by solving certain boundary integral equations. 
\end{enumerate}

Here   above we mentioned paper \cite{Co19} by Costabel. It is worth noticing that the problem considered there  is somehow similar to the one of our paper. The goal is to find a $C^1$-domain $\Omega\subseteq\mathbb{R}^n$ where there is a function $u$, solution of a Poisson equation $\Delta u=g$ with $g\in C^\infty(\overline{\Omega})$ and with either Dirichlet or Neumann homogeneous conditions, and such that  $u$ belongs to $H^{3/2}(\Omega)$ but not to $H^{\sigma+3/2}(\Omega)$ for any $\sigma>0$. Costabel obtains similar results also in the context of $L^p$-Sobolev spaces with $1<p<+\infty$. A predecessor of Costabel's examples, but  for the case where $p=1$,  can be found in Jerison and Kenig \cite[Thm.~1.2 (b)]{JeKe95}. Besides dealing with a different equation, Costabel's view point differs from ours in the sense that his focus is on finding a specific domain, whereas we keep the domain fixed and as regular as possible (viz.~a ball) and look for specific functions.

 In a forthcoming paper we plan to show some application of the results presented here. In particular,  we will  construct some nonlinear transmission problems with solutions that are not continuously differentiable up to the contact interface. In this way we can verify that the regularity of the solutions obtained  in the paper \cite{DaMi15} with Mishuris is, in a sense, optimal.

The paper is organized as follows. Section \ref{prel} is a section of preliminaries. Here we introduce some common notions on spherical harmonics, the trace theorem, Weierstrass and similar functions, and some elements of potential theory. We observe that in this paper we don't need to deal with domains that are less than smooth (that is, $C^\infty$). For this reason some of the general results that we recall are stated with generous regularity assumptions. In Section \ref{sharp} we show that the bounds for $\sigma$ of \eqref{p2} cannot be lowered once \eqref{p1} is verified and are, in this sense, sharp (and similarly for \eqref{p1n=2} and \eqref{p2n=2}). In Section \ref{firstattempts} we present our first attempts to obtain a function that satisfies \eqref{p1} and \eqref{p2}. In Section \ref{random} we try with a random sequence of spherical harmonics and in Section \ref{highest} we obtain a better--but not yet optimal--result with the highest-weight spherical harmonics. In Section \ref{examples} we finally present two examples with the sharp bounds for $\sigma$. In Section \ref{generic} we use a generic sequence of harmonics with the sup-norm/two-norm growth ratio of that grows like $\sqrt{\ln k}$ and in Section \ref{domain} we exploit a specific domain $\Omega$ with a cylinder-like part.

\section{Some preliminaries}\label{prel}

\subsection{Spherical harmonics}

Let $-\Delta_S$ be the positive Laplace-Beltrami operator on the sphere $\mathbb{S}_{n-1}:=\partial\mathbb{B}_n$ with eigenvalues
\[
\lambda_0=0\le\lambda_1\le \lambda_2\le \dots 
\]
and orthnormal eigenfunctions 
\[
w_0, w_1, w_2,\dots
\]
where, as usual, the eigenvalues are  repeated according to their multiplicity and the eigenfunctions are orthonormal with respect to the standard
inner product of $L^2(\mathbb{S}_{n-1})$. Then it is well known (see, e.g., Lions and Magenes \cite[Chap.~1, Remark 7.6]{LiMa72}) that the Sobolev space $H^s(\mathbb{S}_{n-1})$ with $s>0$ can be characterized as the space of distributions $u$ on $\mathbb{S}_{n-1}$ such that $\sum_{j=0}^\infty\lambda_j^s|\langle u,w_j\rangle|^2<+\infty$. Namely, we have
\[
H^s(\mathbb{S}_{n-1})=\left\{u\in\mathcal{D}'(\mathbb{S}_{n-1})\,:\; \sum_{j=0}^\infty\lambda_j^s|\langle u,w_j\rangle|^2<+\infty\right\}\,,
\]
where $\mathcal{D}'(\mathbb{S}_{n-1})$ is the space of distributions on $\mathbb{S}_{n-1}$.

Let $\mathcal{P}_k$ be the space of the homogeneous real polynomials of $x_1,\dots,x_n$ of degree $k\in\mathbb{N}$ (in this paper $\mathbb{N}$ is the set of natural numbers including $0$). Let  $\mathcal{H}_k:=\{P\in\mathcal{P}_k\,:\;\Delta P=0\}$ be the subset of the harmonic polynomials  of $\mathcal{P}_k$, and  $H_k:=\{P_{|\mathbb{S}_{n-1}}\,:\; P\in\mathcal{H}_k\}$ be the set of the restrictions to $\mathbb{S}_{n-1}$ of polynomials of  $\mathcal{H}_k$. The functions of $H_k$ are called spherical harmonics of degree $k$ and we have
\[
d_k:=\dim H_k=(2k+n-2)\frac{(k+n-3)!}{k!(n-2)!} \quad\text{for all } k\in\mathbb{N}\setminus\{0\}
\]
(see, e.g., Stein and Weiss \cite[Chap.~IV, \S2]{StWe71}, see also Folland \cite[Corollary 2.55]{Fo95}). It will be useful to note that 
\begin{equation}\label{dksim}
d_k \sim k^{n-2}\ \text{as }k\to\infty.
\end{equation}
We also observe that if $Y_k\in H_k$ then
\begin{equation}\label{inf<2}
\|Y_k\|_\infty\le \sqrt{\frac{d_k}{\omega_n}}\|Y_k\|_2\,,
\end{equation}
where $\|\cdot\|_\infty$ and $\|\cdot\|_2$ denote the usual norms of the Lebesgue spaces $L^\infty(\mathbb{S}_{n-1})$ and $L^2(\mathbb{S}_{n-1})$, respectively, and $\omega_n$ is the volume of $\mathbb{B}_n$ (see, e.g., Stein and Weiss \cite[Chap.~IV, \S2]{StWe71}, see also Folland \cite[Theorem 2.57]{Fo95}).

Moreover, the eigenvalues of $-\Delta_S$ are given by the sequence 
\[
\mu_k:=k(k+n-2)\quad\text{with }k\in\mathbb{N}
\]
and the eigenspace corresponding to $\mu_k$ is $H_k$ (cf., e.g., Shubin \cite[\S22]{Sh01}). It follows that every $\mu_k$ has multiplicity $d_k$. Now, for all $k\in\mathbb{N}$ let $\{Y_{k,1},\dots,Y_{k,d_k}\}$ be an orthonormal basis for $H_k$. Then we have 
\begin{equation}\label{Hs}
H^s(\mathbb{S}_{n-1})=\left\{u\in\mathcal{D}'(\mathbb{S}_{n-1})\,:\; \sum_{k=0}^\infty(k(k+n-2))^s\sum_{j=1}^{d_k}|\langle u,Y_{k,j}\rangle|^2<+\infty\right\}
\end{equation}
(see also Barcel\'o, Luque, and P\'erez-Esteva \cite[Equation (4)]{BaEtAl20}).

\subsection{The trace theorem}

If $\Omega$ is an open bounded and smooth subset of $\mathbb{R}^n$, then the trace theorem for the Sobolev space $H^s(\Omega)$ with $s>1/2$ can be formulated as follows
\begin{thm}\label{trace}
If $s>1/2$, then there exists a surjective continuous map $\mathrm{tr}$ from $H^s(\Omega)$ to $H^{s-1/2}(\partial\Omega)$ such that 
\[
\mathrm{tr}\, v=u_{|\partial\Omega}
\]
for all functions $v\in H^s(\Omega)$ that have a continuous extension $u$ to $\overline{\Omega}$.
\end{thm}

The  proof can be found in many books, we refer the reader for example to  McLean \cite{Mc00}. For our purpose we  especially need the following easy corollary

\begin{cor}\label{tracecor}
Let $\sigma>0$. If $u$ is a continuous function on $\overline{\mathbb{B}_n}$ and $u_{|\mathbb{S}_{n-1}}$ is not in $H^{\sigma}(\mathbb{S}_{n-1})$, then $u$ is not in $H^{\sigma+1/2}(\mathbb{B}_n)$.
\end{cor}

\subsection{Weierstrass' (and similar) functions}

We collect  some known results on Weierstrass' or  Weierstrass' like functions.  For the sake of completeness we include the proofs, though they can be considered as folklore. We will exploit the following lemma (see also Zygmund \cite[Theorem 4.7]{Zyg02}).

\begin{lem}\label{cj} Let $\sum_k c_k$ be an absolutely convergent real series. Let
\[
f(t):=\sum_{k=0}^\infty c_k\,\cos(kt)\qquad\text{for all }t\in\mathbb{R}\,.
\]
If $f$ belongs to $C^{0,\alpha}(\mathbb{R})$ for some $0<\alpha\le 1$, then there exists $C>0$ such that  $|c_k|\le Ck^{-\alpha}$ for all $k\in\mathbb{N}$.
\end{lem}
\begin{proof} For $k\ge 1$ we have
\[
\begin{split}
c_k&=\frac{1}{\pi}\int_0^{2\pi}f(t)\cos(kt)\,dt=-\frac{1}{\pi}\int_0^{2\pi}f\Bigl(t+\frac{\pi}{k}\Bigr)\cos(kt)\,dt\\
&=\frac{1}{2\pi}\int_0^{2\pi}\left(f(t)-f\Bigl(t+\frac{\pi}{k}\Bigr)\right)\cos(kt)\,dt\,.
\end{split}
\]
Thus,
\[
|c_k|\le \frac{1}{2\pi}\int_0^{2\pi}\left|f(t)-f\Bigl(t+\frac{\pi}{k}\Bigr)\right|\,dt\le \|f\|_{C^{0,\alpha}}\left(\frac{\pi}{k}\right)^\alpha\,.
\]
The statement of the lemma follows.
\end{proof}

Then we have the following result on the Weierstrass' functions. One direction of the proof follows by Lemma \ref{cj}, the other direction is taken from Hardy \cite{Ha16}.

\begin{prop}\label{weierstrass}  Let  $0<\alpha<1$. Let $b\in\mathbb{N}$, $b>1$. Then the function 
\[
f(t):=\sum_{j=0}^\infty b^{-j\alpha}\cos(b^jt)\qquad\text{ for $t\in \mathbb{R}$}
\]
belongs to $C^{0,\alpha}(\mathbb{R})$  and does not belong to $C^{0,\beta}(\mathbb{R})$ for any $\beta>\alpha$.
\end{prop}
\begin{proof} The fact that $f\notin C^{0,\beta}(\mathbb{R})$ for $\beta>\alpha$ follows by Lemma \ref{cj}. Indeed, $f(t)=\sum_{k=0}^\infty c_k\cos(kt)$ with $c_k:=k^{-\alpha}$ if $k=b^j$ for some $j\in\mathbb{N}$ and $c_k:=0$ otherwise.

To prove that $f$ belongs to $C^{0,\alpha}(\mathbb{R})$ we take $-1<\delta<1$ and we write 
\[
\begin{split}
f(t+\delta)-f(t)&=\sum_{j=0}^\infty b^{-j\alpha}\left(\cos(b^jt+b^j\delta)-\cos(b^jt)\right)\\
&=-2\sum_{j=0}^\infty b^{-j\alpha} \sin\Bigl(b^jt+\frac{1}{2}b^j\delta\Bigr)\sin\Bigl(\frac{1}{2}b^j\delta\Bigr)\,,
\end{split}
\]
so that 
\begin{equation}\label{weierstrass.eq1}
\left|f(t+\delta)-f(t)\right|\le 2\sum_{j=0}^\infty b^{-j\alpha} \Bigl|\sin\Bigl(\frac{1}{2}b^j\delta\Bigr)\Bigr|\,.
\end{equation}
Then we remember that $|\delta|<1<b$ and we take $k\in\mathbb{N}$ such that 
\begin{equation}\label{weierstrass.eq1.1}
b^k|\delta|\le 1< b^{k+1}|\delta|\,.
\end{equation}
We split in two the series in the right-hand side of \eqref{weierstrass.eq1} and we write
\begin{equation}\label{weierstrass.eq2}
\begin{split}
\left|f(t+\delta)-f(t)\right| &\le 2\sum_{j=0}^k b^{-j\alpha}  \Bigl|\sin\Bigl(\frac{1}{2}b^j\delta\Bigr)\Bigr|+2\sum_{j=k+1}^\infty b^{-j\alpha}  \Bigl|\sin\Bigl(\frac{1}{2}b^j\delta\Bigr)\Bigr|\\
&\le \sum_{j=0}^k b^{j(1-\alpha)} |\delta|+2\sum_{j=k+1}^\infty b^{-j\alpha}\,.
\end{split}
\end{equation}
For the first term in the right-hand side we have that 
\begin{equation}\label{weierstrass.eq3}
\begin{split}
\sum_{j=0}^k b^{j(1-\alpha)} |\delta|&=\sum_{j=0}^k (b^k |\delta|)^{1-\alpha} (b^{1-\alpha})^{j-k} |\delta|^\alpha\\
&\le  \sum_{j=0}^k  (b^{1-\alpha})^{j-k}|\delta|^\alpha\le \frac{|\delta|^\alpha}{1-b^{\alpha-1}}
\end{split}
\end{equation}
(here we use the first inequality of \eqref{weierstrass.eq1.1}).
For the second term in the right-hand side of \eqref{weierstrass.eq2} we see that 
\begin{equation}\label{weierstrass.eq4}
\sum_{j=k+1}^\infty b^{-j\alpha}=\frac{b^{-(k+1)\alpha}}{1-b^{-\alpha}}< \frac{|\delta|^\alpha}{1-b^{-\alpha}}
\end{equation}
(here we use the second inequality of \eqref{weierstrass.eq1.1}).
Then, by \eqref{weierstrass.eq2}--\eqref{weierstrass.eq4} it follows that $\left|f(t+\delta)-f(t)\right|\le C|\delta|^\alpha$ for all $-1<\delta<1$, where the constant $C>0$ does not depend on $\delta$ and $t$. Since $f$ is bounded (indeed periodic)  this inequality suffices to conclude that $f\in C^{0,\alpha}(\mathbb{R})$.
\end{proof}

As shown by Hardy in \cite{Ha16}, we can relax the assumption that $b$ is a natural number and let $b$ be any real number bigger than $1$. This generalization is however not needed in our paper.   We also observe that, always in \cite{Ha16}, Hardy proved a much stronger result for the function $f$ of Proposition \ref{weierstrass}. Namely he showed that,  for any $\beta>\alpha$, there is no $t\in\mathbb{R}$ such that $f(t+\delta)-f(t)=o(|\delta|^{\beta})$ as $\delta\to 0$.  

By Lemma \ref{cj} we can also deduce the following.

\begin{prop}\label{hardy}  Let $1<b\in\mathbb{N}$. Then the function 
\[
f(t):=\sum_{j=0}^\infty \frac{\cos(b^jt)}{j^2}\qquad\text{ of $t\in \mathbb{R}$}
\]
is continuous and does not belong to $C^{0,\beta}(\mathbb{R})$ for any  $\beta>0$.
\end{prop}
\begin{proof} Indeed, $f(t)=\sum_{k=0}^\infty c_k\cos(kt)$ with $c_k:=(\log_b k)^{-2}$ if $k=b^j$ for some $j\in\mathbb{N}$ and $c_k:=0$ otherwise. Since 
there are no $C>0$ and no $0<\alpha\le 1$ such that  $|c_k|\le Ck^{-\alpha}$ for all $k\in\mathbb{N}$, the statement follows by  Lemma \ref{cj}.
\end{proof}

\subsection{Some potential theory} We  shall also need some basic result of potential theory. Most of the  proofs can be found in monographs that present an introduction to the subject,  for example in Folland \cite[Ch.~3]{Fo95} and \cite{DaLaMu21}. When it is not like that, we will indicate a specific reference. 

We denote by $S_n$ the usual fundamental solution of the Laplace operator on $\mathbb{R}^n$. That is, the function defined by
\[
S_n(x):=
\begin{cases}
\frac{\ln|x|}{2\pi}&\text{if }n=2\,,\\
\frac{|x|^{2-n}}{s_n(2-n)}&\text{if }n\ge 3
\end{cases}
\]
for all $x\in\mathbb{R}^n\setminus\{0\}$, where $s_n$ is the $(n-1)$-dimensional measure of $\mathbb{S}_{n-1}$. Then, if $\Omega$ is a bounded open smooth subset of $\mathbb{R}^n$,  we denote by $\nu_\Omega$ the outward unit normal to $\partial\Omega$ and we set
\[
v[\phi](x):=\int_{\partial\Omega}\phi(y)S_n(x-y)d\sigma_y\qquad\text{for all }x\in\mathbb{R}^n
\]
and 
\[
w[\phi](x):=-\int_{\partial\Omega}\phi(y)\,\nu_{\Omega}(y)\cdot\nabla S_n(x-y)d\sigma_y\qquad\text{for all }x\in\mathbb{R}^n
\]
for all continuous function $\phi$ on $\partial\Omega$. $v[\phi]$ and $w[\phi]$ are respectively the single layer potential and the double layer potential with density $\phi$. We can see that both $v[\phi]$ and $w[\phi]$ are harmonic in $\mathbb{R}^n\setminus\partial\Omega$. Moreover, $v[\phi]$ is continuous on the whole of $\mathbb{R}^n$, whereas $w[\phi]$ has a discontinuity on $\partial\Omega$ (unless $\phi=0$). Nevertheless, the restriction of $w[\phi]$ to $\Omega$ has a continuous extension $w^+[\phi]$ to the closure $\overline{\Omega}$ and the restriction of $w[\phi]$ to $\mathbb{R}^n\setminus\overline\Omega$ has a continuous extension $w^-[\phi]$ to  $\overline{\mathbb{R}^n\setminus\Omega}=\mathbb{R}^n\setminus\Omega$. On the boundary of $\Omega$ we have a ``jump." Namely,
\[
w^+[\phi](x)-w^-[\phi](x)=\phi(x)\quad\text{for all }x\in\partial\Omega
\]
and, if we denote by $W[\phi]$ the restriction of (the discontinuous function) $w[\phi]$ to $\partial\Omega$ we have
\begin{equation}\label{jumps}
w^\pm[\phi]=\pm\frac{1}{2}\phi+W[\phi]\quad\text{on }x\in\partial\Omega\,.
\end{equation}
Schauder \cite{Sc31, Sc32} proved that $\phi\mapsto W[\phi]$ is compact on $C^{0,\alpha}(\partial\Omega)$ for all  $0\le\alpha<1$ 
and thus $\frac{1}{2}I+W$ is a Fredholm operator of index $0$ on $C^{0,\alpha}(\partial\Omega)$. If we further assume that $\mathbb{R}^n\setminus\overline\Omega$ consists of exactly one connected component, than it is possible to prove that $\frac{1}{2}I+W$ is an isomorphism from $C^{0,\alpha}(\partial\Omega)$ to itself. This, together with the jump relations \eqref{jumps} yields 
\begin{prop}\label{dirichlet} Let $0\le\alpha<1$. Let $\Omega$  be a bounded open smooth subset of $\mathbb{R}^n$ such that $\mathbb{R}^n\setminus\overline\Omega$ is connected. Then, for all $f\in C^{0,\alpha}(\partial\Omega)$ the (unique) solution of the Dirichlet problem 
\[
\Delta u=0\text{ in }\Omega\,,\quad u=f\text{ on }\partial\Omega
\]
exists, belongs to $C^{0,\alpha}(\overline\Omega)$, and can be obtained as a double layer potential $w^+[\mu]$ with density   $\mu\in C^{0,\alpha}(\partial\Omega)$.
\end{prop}
\begin{proof} Since  $\mathbb{R}^n\setminus\overline\Omega$ is connected $\frac{1}{2}I+W$ is an isomorphism from $C^{0,\alpha}(\partial\Omega)$ to itself and thus there exists $\mu\in C^{0,\alpha}(\partial\Omega)$ such that $w^+[\mu]_{|\partial\Omega}=\frac{1}{2}\mu+W[\mu]=f$. Then $w^+[\mu]$ is the unique solution of the Dirichlet boundary value problem (uniqueness comes from the maximum principle) and Miranda \cite[Theorem 2.I]{Mi65} proves that $w^+[\mu]$ belongs to $C^{0,\alpha}(\overline\Omega)$ (see also \cite[Theorem 4.17]{DaLaMu21} for a ``modern'' presentation of Miranda's result).
\end{proof}
In Proposition \ref{dirichlet} the assumption that $\mathbb{R}^n\setminus\overline\Omega$ is connected could be dropped if we write the solution as a combination of a double layer and a single layer potentials (see \cite{Fo95} and \cite{DaLaMu21}). In this paper we don't need such generality though. 

\section{The bounds for $\sigma$ are sharp}\label{sharp}

Here we see that the bounds $\sigma>0$ and $\sigma\ge \alpha$ in condition \eqref{p2} are sharp once the function $u$ satisfies condition \eqref{p1}. To deal with the case where $\alpha>0$, we use the following classical lemma (see, for example,  \cite[Lemma 8.3.(i)]{Wo03}). We include a proof for the sake of completeness. 

\begin{lem}\label{CinH} Let $0<\alpha<1$. If $\Omega$ is a smooth  open and bounded  subset of $\mathbb{R}^n$, then $C^{0,\alpha}(\partial\Omega)\subset H^\sigma(\partial\Omega)$ for all $0<\sigma<\alpha$.
\end{lem}
\begin{proof} Let $v$ be a function of $C^{0,\alpha}(\partial\Omega)$. Let $d:=\mathrm{diam}\,\Omega$ be the diameter of $\Omega$, let $\mathbb{B}_n(y,r):=y+r\mathbb{B}_n$ be the open ball with center $y$ and radius $r>0$ and let $\mathbb{A}_n(y,r,R)$ be the open annulus $\mathbb{B}_n(y,R)\setminus\overline{\mathbb{B}_n(y,r)}$. Then we have
\[
\begin{split}
&\int_{\partial\Omega}\int_{\partial\Omega}\frac{|v(x)-v(y)|^2}{|x-y|^{n-1+2\sigma}}d\sigma_xd\sigma_y\lesssim \int_{\partial\Omega}\int_{\partial\Omega}\frac{1}{|x-y|^{n-1+2(\sigma-\alpha)}}d\sigma_xd\sigma_y\\
&\quad=\int_{\partial\Omega}\sum_{j=0}^{+\infty}\int_{\partial\Omega\cap\mathbb{A}_n(y,d/2^{j+1},d/2^j)}\frac{1}{|x-y|^{n-1+2(\sigma-\alpha)}}d\sigma_xd\sigma_y\\
&\quad\lesssim \int_{\partial\Omega}\sum_{j=0}^{+\infty}2^{j(n-1+2(\sigma-\alpha))}\left|\partial\Omega\cap\mathbb{B}_n(y,d/2^j)\right|d\sigma_y\\
&\quad\lesssim \sum_{j=0}^{+\infty}2^{j(n-1+2(\sigma-\alpha))}(d/2^j)^{n-1}= \sum_{j=0}^{+\infty}2^{2j(\sigma-\alpha)}
\end{split}
\]
and the last series converges as soon as $\sigma<\alpha$. In the last inequality we have used the fact that $\left|\partial\Omega\cap\mathbb{B}_n(y,r)\right|\lesssim r^{n-1}$ when the  boundary $\partial\Omega$ is smooth enough.
\end{proof}

Then we deduce the following.

\begin{prop}\label{prop:sharp} Let $\Omega$  be an open bounded smooth  subset of $\mathbb{R}^n$. If $u$ is harmonic in $\Omega$ and continuous on $\overline\Omega$, then it belongs to $H^{1/2}(\Omega)$. If $u$ is harmonic in $\Omega$ and belongs to $C^{0,\alpha}(\overline\Omega)$ for some $0<\alpha<1$, then it belongs to $H^{\sigma+1/2}(\Omega)$  for all $0<\sigma<\alpha$.
\end{prop}
\begin{proof} For $\alpha=0$ the statement follows by the regularity result for very weak solutions of the Dirichlet problem of Ne\v{c}as \cite[Chap.~6.1.2, Thm.~1.3]{Ne12} and by the characterization of harmonic functions in $H^{1/2}(\Omega)$ of Jerison and Kenig \cite[Thm.~4.1]{JeKe95}. For the case where $\alpha>0$, we see that   $u=w^+[\mu]$ for some $\mu\in C^{0,\alpha}(\partial\Omega)$ by Proposition \ref{dirichlet}. Then the statement is a consequence of Lemma \ref{CinH} and of the mapping properties of the double layer potential (see, e.g., Costabel \cite[Thm.~1]{Co88}).
\end{proof}

\section{First attempts}\label{firstattempts}

\subsection{With a random sequence of spherical harmonics}\label{random}

Inspired by the work of Hadamard \cite{Ha06} we imitate formula  \eqref{hadamard.1} and we look at functions in the form 
\begin{equation}\label{idea}
u(x):=\sum_{k=0}^\infty{a_k}r^k\, Y_k(\theta)\quad\text{for }x\in \overline{\mathbb{B}_n}\setminus\{0\}\qquad\text{and }u(0):=a_0Y_0\,,
\end{equation}
where $\{Y_k\}_{k}$ is a sequence of spherical harmonics with $Y_k\in H_k$ (we recall that $Y_0$ is a constant function) and $\{a_k\}_k$ is a suitable real sequence. As usual, $r=|x|$ and $\theta=x/|x|$. Our goal is to choose $\{Y_k\}_{k}$ and $\{a_k\}_k$ such that the series in \eqref{idea} converges normally with respect to $\|\cdot\|_\infty$, that is 
\begin{equation}\label{1st}
\sum_{k=0}^\infty |a_k|\|Y_k\|_\infty\text{ converges,}
\end{equation}
and, simultaneously, the series 
\[
\sum_{k=0}^\infty(k(k+n-2))^\sigma\sum_{j=1}^{d_k}|\langle  u_{|\mathbb{S}_{n-1}},Y_{k,j}\rangle|^2
\]
does not converge for $\sigma$ big enough. Let's say, for $\sigma>\sigma_0$. Here above  $\{Y_{k,1},\dots,Y_{k,d_k}\}$ is an orthonormal basis of $H_k$ and we can assume that 
 \[
 Y_{k,1}=\frac{Y_k}{\|Y_k\|_2}\,.
 \]
 Then we can rewrite the second condition  as
 \begin{equation}\label{2nd}
 \sum_{k=0}^\infty k^{2\sigma} a_k^2 \|Y_{k}\|_2^2\quad\text{does not to converge for $\sigma>\sigma_0$.} 
 \end{equation}
If conditions \eqref{1st} and \eqref{2nd} are verified, then $u$ is a continuous harmonic function on $\overline{\mathbb{B}_n}$ and, by \eqref{Hs}, $u_{|\mathbb{S}_{n-1}}\notin H^{\sigma}(\mathbb{S}_{n-1})$ for any $\sigma>\sigma_0$. Then we conclude by Corollary \ref{tracecor} that $u$ is not in $H^{\sigma+1/2}(\mathbb{B}_n)$ for any $\sigma>\sigma_0$. 

Our goal is to have the bound $\sigma_0$ as small as possible and clearly the ratio between  the sup-norm $\|Y_k\|_\infty$ and the two-norm $\|Y_k\|_2$ plays an important role in this task.  With the exception of the case where $n=2$, such ratio is expected to grow when $k\to\infty$. Then, to have $\sigma_0$ small, we have to chose a sequence $\{Y_k\}_k$ for which the ratio grows as slow as possible.  By the asymptotic formula \eqref{dksim} and by inequality \eqref{inf<2}, we can see that for any sequence of  spherical harmonics $\{Y_k\}_k$ with $Y_k\neq 0$ we  have
\begin{equation}\label{anyYk}
\frac{\|Y_k\|_\infty}{\|Y_k\|_2}=O(k^{\frac{n-2}{2}})\quad\text{as }k\to\infty\,,
\end{equation}
which, for our purposes, is very good when $n=2$, but gets worse when $n$ is bigger. The growth relation \eqref{anyYk} is attained by the zonal spherical harmonics $Z_k$, for which 
\[
\frac{\|Z_k\|_\infty}{\|Z_k\|_2}\sim k^{\frac{n-2}{2}}\quad\text{as }k\to\infty\,.
\]
If in \eqref{idea} we use one such a sequence of spherical harmonics we obtain the following proposition, where we find convenient to normalize the sequence of spherical harmonics with respect to the infinity norm. Also, we modify the sequence \eqref{hadamard.2} of Hadamard to obtain a better bound for $\sigma$. 
 
\begin{prop}\label{withanyYk} Let $\{Y_k\}_k$ be a sequence of spherical harmonics with $Y_k\in H_k$ and $\|Y_k\|_\infty=1$. Let
\begin{equation}\label{ak}
a_k:=
\begin{cases}
0&\text{if $k\neq 2^j$ for all $j\in \mathbb{N}\setminus\{0\}$}\,,\\
j^{-2}&\text{if $k= 2^j$ for some $j\in \mathbb{N}\setminus\{0\}$}\,.
\end{cases}
\end{equation}
 Let $u$ be as in \eqref{idea} with $\{a_k\}_k$ as above. Then  $u$ is harmonic in ${\mathbb{B}_n}$,  continuous on $\overline{\mathbb{B}_n}$, and does not belong to  $H^{\sigma+1/2}(\mathbb{B}_n)$ for any $\sigma>(n-2)/2$. Moreover, $u$ does not belong to $C^{0,\alpha}(\overline{\mathbb{B}_n})$ for any $\alpha>(n-2)/2$.
\end{prop}
\begin{proof} It suffices to check that  conditions \eqref{1st} and \eqref{2nd} hold with $\sigma_0=(n-2)/2$.  For \eqref{1st} we observe that $\sum_{k=0}^\infty |a_k|\|Y_k\|_\infty$ is smaller than a constant times
\[
\sum_{j=1}^\infty\frac{1}{j^2}=\frac{\pi^2}{6}<+\infty\,.
\]
(Incidentally, we remember that $\sum_{j=1}^\infty1/j^2=\pi^2/6$ is the  Basel problem sum.)
For \eqref{2nd},  we see that  $\sum_{k=0}^\infty k^{2\sigma} a_k^2 \|Y_k\|_2^2$ is bigger than a constant times
\[
\sum_{j=0}^\infty \frac{2^{\left(2\sigma-n+2\right)j}}{j^4}\,,
\]
and the last series diverges as soon as $\sigma>(n-2)/2$. The last statement of the proposition is a consequence of 
Lemma \ref{CinH}. \end{proof}

If we had used the sequence $\{a_k\}_k$ in \eqref{hadamard.2},  we would have obtained $(n-1)/2$ as a bound for $\sigma$, which is bigger (and thus worse) than the bound $(n-2)/2$ obtained with \eqref{ak}. That having been said, the bound $(n-2)/2$ is fine only  for $n=2$ and is far from the expected sharp bound when $n\ge 3$. Also,  it is not clear how we can  check if $u$ belongs to $C^{0,\alpha}(\overline{\mathbb{B}_n})$ for some positive $\alpha$. Lemma \ref{CinH} only grants that $u$ is not in $C^{0,\alpha}(\overline{\mathbb{B}_n})$ for $\alpha$ big enough. 

However, there is a better result if we replace the random sequence $\{Y_k\}_k$ of Proposition \ref{withanyYk} with the sequence of highest-weight spherical harmonics. 

\subsection{With the highest-weight spherical harmonics}\label{highest}

The sequence $\{Q_k\}_k$ of highest-weight spherical harmonics is defined by 
\[
 Q_k(\theta):=\mathfrak{R}(\theta_1+i\theta_2)^k\quad\text{for all }\theta=(\theta_1,\dots,\theta_n)\in\mathbb{S}_{n-1}
\]
(see Sogge \cite{So15}). For all $k\in\mathbb{N}$ we have that $Q_k\in H_k$ and clearly it is 
\[
\| Q_k\|_\infty=1\,.
\]
We can also see that
\[
\| Q_k\|_2^2=\frac{\pi^{\frac{n}{2}}k!}{\Gamma\bigl(k+\frac{n}{2}\bigr)}
\]
(the reader can find in Folland \cite{Fo01} some helpful hint  for this computation)\,.
Then, using the Stirling formula 
\[
\Gamma(t+1)\sim\sqrt{t}\left(\frac{t}{e}\right)^t\quad\text{for }t\to+\infty\,,
\]
we can verify that 
\begin{equation}\label{asymptotic}
\| Q_k\|_2\sim k^{\frac{2-n}{4}}\quad\text{for }k\to+\infty\,.
\end{equation}
So, for the sequence $\{Q_k\}_k$  the ratio $\| Q_k\|_\infty/\| Q_k\|_2$ grows like 
\[
k^{\frac{n-2}{4}}
\]
as $k\to \infty$, which is a slower rate with respect to \eqref{anyYk}.
With the sequence $\{Q_k\}_k$ we can prove the following
\begin{prop}\label{notCbeta>0} Let $\{a_k\}_k$ be defined as in \eqref{ak}. Let $u$ be the function from $\overline{\mathbb{B}_n}$ to $\mathbb{R}$ defined by
\begin{equation}\label{ux.2}
u(x):=\sum_{k=0}^\infty a_k\,\mathfrak{R}(x_1+ix_2)^k=\sum_{j=1}^\infty \frac{1}{j^2}\,\mathfrak{R}(x_1+ix_2)^{2^j}
\end{equation}
for all $x=(x_1,\dots,x_n)\in\overline{\mathbb{B}_n}$. Then $u$ is harmonic in ${\mathbb{B}_n}$,  continuous on $\overline{\mathbb{B}_n}$, does not belong to $C^{0,\epsilon}(\overline{\mathbb{B}_n})$ for any $\epsilon>0$, and does not belong to $H^{\sigma+1/2}(\mathbb{B}_n)$ for any $\sigma>(n-2)/4$.
\end{prop}
\begin{proof}  
To prove that $u$ is harmonic, continuous on $\overline{\mathbb{B}_n}$, and  not in $H^{\sigma+1/2}(\mathbb{B}_n)$ for $\sigma>(n-2)/4$ we check that conditions \eqref{1st} and \eqref{2nd} are verified with 
\[
Y_k(\theta)=Q_k(\theta)=\mathfrak{R}(\theta_1+i\theta_2)^k\quad\text{for all }\theta=(\theta_1,\dots,\theta_n)\in\mathbb{S}_{n-1}\,,
\]
$\{a_k\}_k$ as in \eqref{ak}, and $\sigma_0=(n-2)/4$.  For \eqref{1st} we observe that 
\begin{equation}\label{notCbeta>0.eq1}
\sum_{k=0}^\infty |a_k|\|Q_k\|_\infty=\sum_{j=1}^\infty\frac{1}{j^2}=\frac{\pi^2}{6}<+\infty\,.
\end{equation}
For \eqref{2nd},  we see that  
\[
 \sum_{k=0}^\infty k^{2\sigma} a_k^2 \|Q_{k}\|_2^2= \sum_{j=1}^\infty \frac{ 2^{2\sigma j}}{j^4} \|Q_{2^j}\|_2^2
\]
and by \eqref{asymptotic} the $j$-{th} term of the last series is asymptotic to 
\[
\frac{2^{\left(\sigma-\frac{n-2}{4}\right)2j}}{j^4}
\]
as $j\to\infty$. Since the quotient above
diverges for $\sigma>(n-2)/4$, condition \eqref{2nd} is verified.

It remains to prove that $u$ is not in  $C^{0,\epsilon}(\overline{\mathbb{B}_n})$ for any $\epsilon>0$. Let $\psi$ be the map from $\mathbb{R}$ to $\mathbb{S}_{n-1}$ that takes $t$ to 
\[
(\cos t,\sin t,0,\dots,0)\,.
\]
By Proposition \ref{hardy} (and by a standard argument based on Euler formula) we can verify that the composition $u\circ\psi$ 
does not belong to $C^{0,\epsilon}(\mathbb{R})$ for any $\epsilon>0$. Since $\psi$ is smooth (real analytic indeed), we conclude that   $u$ is not in $C^{0,\epsilon}(\overline{\mathbb{B}_n})$.
\end{proof}

We observe that in the proof of Proposition \ref{notCbeta>0} we did not use Lemma \ref{CinH} to show that $u$ is not in $C^{0,\epsilon}(\overline{\mathbb{B}_n})$, this would yield the result for $\epsilon>(n-2)/4$ and not $\epsilon>0$. Also, we might have been tempted to adopt a different strategy. Namely, to prove first that the function $u$ is not in $C^{0,\epsilon}(\mathbb{R})$ for any $\epsilon>0$, and then to recover the result concerning the membership in $H^{\sigma+1/2}(\mathbb{B}_n)$ by the  Morrey inequality for fractional order Sobolev spaces (see, e.g., Di Nezza, Palatucci, and Valdinoci \cite[Theorem 8.2]{DiEtAl12}). Doing so, however, we would obtain that $u$ is not in $H^{\sigma+1/2}(\mathbb{B}_n)$ for $\sigma>(n-1)/2$, and the bound $(n-1)/2$ is  worse than the bound $(n-2)/4$ that we find.

With different sequences $\{a_k\}_k$ we can obtain harmonic functions that are in $C^{0,\alpha}(\overline{\mathbb{B}_n})$ for some positive $\alpha$ but not in $C^{0,\alpha+\epsilon}(\overline{\mathbb{B}_n})$ for $\epsilon>0$ and are not in $H^{\sigma+1/2}(\mathbb{B}_n)$ for $\sigma\ge \alpha+(n-2)/4$. We show this in the following theorem, where we exploit Theorem  \ref{weierstrass} on the Weierstrass-like functions.

\begin{prop}\label{notCbeta>alpha} Let $0<\alpha<1$. Let
\begin{equation}\label{ak.holder}
a_k:=
\begin{cases}
0&\text{if $k\neq 2^j$ for all $j\in \mathbb{N}\setminus\{0\}$}\,,\\
2^{-j\alpha}&\text{if $k= 2^j$ for some $j\in \mathbb{N}\setminus\{0\}$}\,.
\end{cases}
\end{equation}
Let $u$ be the function from $\overline{\mathbb{B}_n}$ to $\mathbb{R}$ defined by
\begin{equation}\label{ux.holder}
u(x):=\sum_{k=0}^\infty a_k\,\mathfrak{R}(x_1+ix_2)^k=\sum_{j=1}^\infty \frac{1}{2^{j\alpha}}\,\mathfrak{R}(x_1+ix_2)^{2^j}
\end{equation}
for all $x=(x_1,\dots,x_n)\in\overline{\mathbb{B}_n}$. Then $u$ is harmonic in ${\mathbb{B}_n}$,  H\"older continuous with exponent $\alpha$ on $\overline{\mathbb{B}_n}$, does not belong to $C^{0,\alpha+\epsilon}(\overline{\mathbb{B}_n})$ for any $\epsilon>0$, and does not belong to $H^{\sigma+1/2}(\mathbb{B}_n)$ for any $\sigma\ge\alpha+(n-2)/4$.
\end{prop}
\begin{proof} As in the proof of Proposition \ref{notCbeta>0} we begin by noting that 
\begin{equation}\label{notCbeta>alpha.eq1}
\sum_{k=0}^\infty |a_k|\|Q_k\|_\infty=\sum_{j=1}^\infty\frac{1}{2^{j\alpha}}=\frac{1}{2^\alpha-1}<+\infty\,.
\end{equation}
Thus the series in \eqref{ux.holder} converges normally  with respect to $\|\cdot\|_\infty$ and the function $u(x)$ is harmonic in $\mathbb{B}_n$ and continuous in $\overline{\mathbb{B}_n}$. 

Next, by Proposition \ref{weierstrass} (and by a standard argument based on Euler formula) we see that the map from $\mathbb{S}_1$ to $\mathbb{R}$ that takes $(x_1,x_2)$  to $\sum_{k=0}^\infty a_k\,\mathfrak{R}(x_1+ix_2)^k$ is H\"older continuous with exponent $\alpha$.  We also note that $(x_1,x_2)\mapsto \sum_{k=0}^\infty a_k\,\mathfrak{R}(x_1+ix_2)^k$ is harmonic and, since the solution of the Dirichlet problem in $\mathbb{B}_2$ with a boundary datum in $C^{0,\alpha}(\mathbb{S}_1)$ exists, is unique, and of class $C^{0,\alpha}(\overline{\mathbb{B}_2})$, we deduce that for $n=2$ the function $u$ belongs to $C^{0,\alpha}(\overline{\mathbb{B}_2})$. Since, for $n\ge 3$, $u$ does not depend on $(x_3,\dots,x_n)$, we conclude that $u$ belongs to $C^{0,\alpha}(\overline{\mathbb{B}_n})$ for all $n\ge 2$. To prove that $u$ does not belong to $C^{0,\alpha+\epsilon}(\overline{\mathbb{B}_n})$ for $\epsilon>0$, we argue as we did in the proof of Theorem \ref{notCbeta>0} to prove that $u$ is not in $C^{0,\epsilon}(\overline{\mathbb{B}_n})$ for $\epsilon>0$, we only have to use Proposition \ref{weierstrass} instead of Proposition \ref{hardy}. 

Finally, to verify that $u\notin H^{\sigma+1/2}(\mathbb{B}_n)$ for $\sigma\ge \alpha+(n-2)/4$ it remains to check condition \eqref{2nd} (with $\sigma\ge \sigma_0$ instead of $\sigma>\sigma_0)$. We see that  
\[
 \sum_{k=0}^\infty k^{2\sigma} a_k^2 \|Q_{k}\|_2^2= \sum_{j=1}^\infty 2^{2j(\sigma-\alpha)} \|Q_{2^j}\|_2^2
\]
and by \eqref{asymptotic} the $j$-{th} term of the last series is asymptotic to 
\[
\left(2^{\left(\sigma-\alpha-\frac{n-2}{4}\right)}\right)^{2j}
\]
as $j\to\infty$. Then condition \eqref{2nd} is verified  for $\sigma-\alpha-\frac{n-2}{4}\ge 0$. That is, $\sigma\ge \alpha+(n-2)/4$.
 \end{proof}
 
 The bounds for $\sigma$ obtained in Propositions \ref{notCbeta>0} and \ref{notCbeta>alpha} are better than those in Proposition \ref{withanyYk}, but still, they are sharp only when $n=2$. Also, using the highest-weight spherical harmonics we could keep  the H\"older continuity in check. 
 
 In the following section we see two different ways to obtain the sharp bounds.

\section{Two examples}\label{examples}

\subsection{With a generic sequence of spherical harmonics}\label{generic} We have seen in Section \ref{firstattempts} that the bounds for $\sigma$ are related to  the ratio between the sup-norm and the the two-norm of the sequence $\{Y_k\}_k$ that we use in the construction of the function $u$. The slower this ratio grows as $k$ tends to $\infty$, the smaller, and thus better, is the bound for $\sigma$. For a random sequence of spherical harmonics we can only say that the ratio grows slower than $k^{\frac{n-2}{2}}$. The situation is a little better for the highest-weight spherical harmonics, whose ratio grows like $k^{\frac{n-2}{4}}$.  Fortunately, the growth rate of a generic sequence of spherical harmonics is even better. More precisely,  a result of Neuheisel \cite{Ne00} shows that for almost all sequences of spherical harmonics the ratio between the sup-norm and the two-norm grows like $\sqrt{\ln k}$.

In order to present properly  the result of Neuheisel we have  to introduce some notation. So, we equip the (finite dimensional) space $H_k$ with the norm induced by $L^2(\mathbb{S}_{n-1})$ and we denote by $\mathbb{S}_{H_k}$ the corresponding unit ball in $H_k$. Namely, $\mathbb{S}_{H_k}$ consists  of the spherical harmonics $Y\in H_k$ with $\|Y\|_2=1$. Then $\mathbb{S}_{H_k}$ is a sphere in a space of dimension $d_k$ and we can define on it the usual Lebesgue measure $m_k$, which we normalize to be a probability measure:
\[
m_k(\mathbb{S}_{H_k})=1\,.
\]
We also set $\mathbb{S}_{H_\infty}:=\prod_{k=1}^\infty \mathbb{S}_{H_k}$ and $m_\infty:=\prod_{k=1}^\infty m_k$. Then we have the following.

\begin{thm}[Neuheisel \cite{Ne00}]\label{Neu} For $m_\infty$-almost all sequences $\{Y_k\}_k$ with $Y_k\in \mathbb{S}_{H_k}$ we have
\begin{equation}\label{Neu.eq1}
\|Y_k\|_\infty=O(\sqrt{\ln k})\quad\text{as }k\to\infty\,.
\end{equation}
\end{thm}

Theorem \ref{Neu} has been extended in different directions. We mention, for example, the work of Feng and Zelditch \cite{FeZe14} on random holomorphic fields on  K\"ahler manifolds and the paper of Canzani and Hanin \cite{CaHa15} in the framework of compact Riemannian manifolds (see also the references therein).

Using a sequence $\{Y_k\}_k$ as in Theorem \ref{Neu},   we can now construct a harmonic function that is continuous in $\overline{\mathbb{B}_n}$  and does not belong to $H^s(\mathbb{B}_n)$ for any $s>1/2$. 

\begin{thm}\label{notHs} Let $u$ be as in \eqref{idea} with $\{a_k\}_k$ as in \eqref{ak} and with a sequence of spherical harmonics $\{Y_k\}_k$ such that  $Y_k\in\mathbb{S}_{H_k}$ (i.e. $Y_k\in H_k$ and $\|Y_k\|_{L^2(\partial\mathbb{B}_n)}=1$) and \eqref{Neu.eq1} holds true.   Then $u$ is harmonic in ${\mathbb{B}_n}$,  continuous on $\overline{\mathbb{B}_n}$, and does not belong to  $H^{\sigma+1/2}(\mathbb{B}_n)$ for any $\sigma>0$. Moreover, $u$ does not belong to $C^{0,\epsilon}(\overline{\mathbb{B}_n})$ for any $\epsilon>0$.
\end{thm}
\begin{proof} The existence of $\{Y_k\}_k$ is granted by Theorem \ref{Neu}. Then, we have to check that conditions \eqref{1st} and \eqref{2nd} are verified with $\{Y_k\}_k$ as in \eqref{Neu.eq1},
$\{a_k\}_k$ as in \eqref{ak}, and $\sigma_0=0$.  For \eqref{1st} we observe that $\sum_{k=0}^\infty |a_k|\|Y_k\|_\infty$ is smaller than (a constant times)
\[
\sum_{j=1}^\infty\frac{\sqrt{\ln 2^j}}{j^2}=\sqrt{\ln 2}\sum_{j=1}^\infty\frac{1}{j^{3/2}}<+\infty\,.
\]
For \eqref{2nd},  we see that  
\[
 \sum_{k=0}^\infty k^{2\sigma} a_k^2 \|Y_k\|_2^2= \sum_{j=0}^\infty \frac{ 2^{2\sigma j}}{j^4}\,.
\]
The last series diverges as soon as $\sigma>0$. We deduce that $u$ is harmonic in ${\mathbb{B}_n}$,  continuous on $\overline{\mathbb{B}_n}$, and does not belong to  $H^{\sigma+1/2}(\mathbb{B}_n)$ for any $\sigma>0$. Then the last statement on the H\"older continuity is a consequence of Lemma \ref{CinH}.
\end{proof}

\subsection{Passing through a different domain}\label{domain} Let us now think about the function $u$ of Proposition \ref{notCbeta>0} or, equivalently, that of Proposition  \ref{notCbeta>alpha}.  We can see that for $n\ge 3$ the series 
\[
\sum_{k=0}^\infty a_k\,\mathfrak{R}(x_1+ix_2)^k\quad\text{with }x=(x_1,x_2,\dots,x_n)\in\mathbb{R}^n
\]
defines a harmonic function on the infinite open cylinder
\[
\mathbb{B}_2\times\mathbb{R}^{n-2}\,.
\]
So, the function $u$, which is defined on the closed ball $\overline{\mathbb{B}_n}$, is actually everywhere real analytic on its domain with the exception of the equator line $\mathbb{S}_1\times(0,\dots,0)$. We might say that the singularity of $u$ is solely localized on that equator line. Since we are looking for continuous harmonic functions that are, in a sense, as singular as possible, it seems a good idea to consider an extension of $u$ to a domain $\Omega$ that contains an open portion of the cylinder surface
\[
\mathbb{S}_1\times\mathbb{R}^{n-2}\,.
\]
The idea of using such a set $\Omega$ is taken from a paper of Costabel \cite{Co19}. To define $\Omega$, 
we use a function $\mu\in C^\infty(\overline{\mathbb{R}_+})$ such that $\mu=1$ in $[0,R]$ for some $R>0$, $\mu>0$ in $(0,2R)$, $\mu<0$ in $(2R,\infty)$, and $\mu'(2R)<0$.  We may for example take
\begin{equation}\label{supermu}
\mu(t):=
\begin{cases}
1&\text{for }0\le t\le R\,,\\
1-e^{\frac{t-2R}{t-R}}&\text{for }t>R\,.
\end{cases}
\end{equation}
Then we set
\[
\Omega:=\left\{ (x_1, x_2, x'')\in\mathbb{R}^n : x_1^2+x_2^2<  \mu(|x''|^2) \right\}.
\]
\begin{figure}
\center
\begin{tikzpicture}         
  \begin{axis}[
  	colormap={blackwhite}{gray(0cm)=(0); gray(1cm)=(1)},
        axis equal image,
        grid = both,
        axis lines = center,
      ticks=none,
        xlabel = {$x_1$},
        ylabel = {$x_2$},
        zlabel = {$x_3$},
        major grid style = {draw = lightgray},
        minor grid style = {draw = lightgray!25},
        legend cell align={left},
        xmin = -2.5, xmax = 2.5,
        ymin = -4.2, ymax = 4.2,
        scale = 3,
        zmin = -5, zmax = 5,
        z buffer = sort,
        every axis/.append style={font=\tiny},
    ]
  	\addplot3[
	opacity=0.9,
            surf,
            shader=interp,
            samples=50,
            domain=-2:2, y domain=0:2*pi,
            z buffer=sort]
        ({1 * cos(deg(y))}, {1 * sin(deg(y))},\x);
        \addplot3[
        opacity=0.8,
            surf,
            shader=interp,
            samples=50,
            domain=2:4, y domain=0:2*pi,
            z buffer=sort]
        ({(1/sqrt(12))*sqrt(16-\x^2) * cos(deg(y))}, {(1/sqrt(12))*sqrt(16-\x^2) * sin(deg(y))},\x); 
         \addplot3[
         opacity=0.7,
            surf,
            shader=interp,
            samples=50,
            domain=-4:-2, y domain=0:2*pi,
            z buffer=sort]
        ({(1/sqrt(12))*sqrt(16-\x^2) * cos(deg(y))}, {(1/sqrt(12)))*sqrt(16-\x^2) * sin(deg(y))},\x); 
        
        \draw[
    samples = 50,
    smooth,
    domain = 0:2,
    variable = \t,
    thick
]
plot (
   {1},{0},{\t});

          \draw[
    samples = 50,
    smooth,
    domain = 2:4,
    variable = \t,
    thick
]
plot (
 {(1/sqrt(12))*sqrt(16-\t^2)},{0}, {\t});
   
\draw[
    samples = 50,
    smooth,
    domain =0:pi,
    variable = \t,
    dashed
]
plot (
   {cos(deg(\t))},{sin(deg(\t))}, {2});
   
   \draw[
    samples = 50,
    smooth,
    domain = pi:2*pi,
    variable = \t,
]
plot (
      {cos(deg(\t))},{sin(deg(\t))},{2});

   \draw[
    samples = 50,
    smooth,
    domain = 0:pi,
    variable = \t,
    dashed
]
plot(
   {cos(deg(\t))},{sin(deg(\t))},{-2});
   
\draw[
    samples = 50,
    smooth,
    domain = pi:2*pi,
    variable = \t,
]
plot (
   {cos(deg(\t))},{sin(deg(\t))},{-2});

   \draw[
    samples = 50,
    smooth,
    domain = 0:pi,
    variable = \t,
    dashed
]
plot(
   {cos(deg(\t))},{sin(deg(\t))},{0});
   
     \draw[
    samples = 50,
    smooth,
    domain = pi:2*pi,
    variable = \t,
]
plot (
   {cos(deg(\t))},{sin(deg(\t))},{0});
   
  \end{axis}    
   \end{tikzpicture}

\caption{{\it The set  $\Omega$ for  $n=3$ and for a suitable choice of $\mu$. The line on the side is the function $\sqrt{\mu(x_3^2)}$}}\label{fig}
\end{figure}
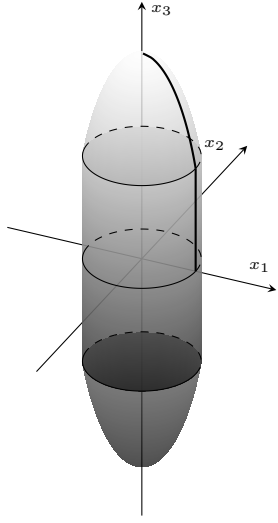 
We can verify that $\Omega$ is open, bounded, and of class $C^\infty$ (see Figure \ref{fig}). Since $\mu\leq1$ we have that 
\[
\Omega\subset \left\{(x_1,x_2,x'')\in\mathbb{R}^n : x_1^2+x_2^2<1 \right\}=\mathbb{B}_2\times\mathbb{R}^{n-2}.
\]
Moreover, the boundary $\partial\Omega=\{ x_1^2+x_2^2=\mu(|x''|^2) \}$ contains an open subset of $\mathbb{S}_1\times\mathbb{R}^{n-2}$. In particular, 
\begin{equation}\label{delta}
\mathbb{S}_1\times \left[-\delta,\delta\right]^{n-2}\subset\partial\Omega
\end{equation}
for $0<\delta\le\sqrt{{R}/({n-2})}$.
Then in $\Omega$ we can prove the following. 
\begin{thm}\label{inOmega} Let $n\ge 3$ and let $\Omega$ be as above. Let $0\le\alpha<1$. Let
\[
u(x):=\sum_{k=0}^\infty a_k\Re(x_1+ i x_2)^k\quad\text{for all }x=(x_1,x_2,\dots,x_n)\in\overline\Omega\,,
\] 
with $\{a_k\}_k$ as in \eqref{ak} if $\alpha=0$ and $\{a_k\}_k$ as in \eqref{ak.holder} if $0<\alpha<1$. Then $u$ is harmonic in $\Omega$, belongs to  $C^{0,\alpha}(\overline\Omega)$, does not belong to $C^{0,\alpha+\epsilon}(\overline\Omega)$ for any $\epsilon>0$,  and does not belong to $H^{\sigma+1/2}(\Omega)$ for any $\sigma>0$ if $\alpha=0$ and any $\sigma\ge \alpha$ if $0<\alpha<1$.
\end{thm}
In other words, $u$ satisfies conditions \eqref{p1} and \eqref{p2} with $\Omega$ replacing $\mathbb{B}_n$.  To prove Theorem \ref{inOmega} we use an elementary lemma on the trace of Sobolev functions that are constant along a direction. In what fallows we shall denote by $\mathbb{T}^m$ the torus 
\[
\mathbb{T}^m:=\mathbb{R}^m/(2\pi\mathbb{Z})^m
\]
of dimension $m\in\mathbb{N}\setminus\{0\}$, and we understand that a function on $\mathbb{T}^m$ extends to a periodic function on $\mathbb{R}^m$.

\begin{lem}\label{torustrace}
Let $m', m''\in\mathbb{N}\setminus\{0\}$. Let  $\tilde{v}$ be a function on  $\mathbb{T}^{m'}$ and  let $v(x',x''):=\tilde v(x')$ for all   $(x',x'')\in\mathbb{T}^{m'+m''}=\mathbb{T}^{m'}\times\mathbb{T}^{m''}$ (so that  $v$ is constant in the direction $x''$). Then  $v$ belongs to $H^s(\mathbb{T}^{m'+m''})$ for some $s\ge 0$ if and only if $\tilde v$ belongs to $H^s(\mathbb{T}^{m'})$ for the same $s$.
\end{lem}
\begin{proof} We can see that a function $w$ on $\mathbb{T}^m$ belongs to $H^s(\mathbb{T}^m)$ if and only if the sum 
\[
\sum_{z\in\mathbb{Z}^m}(1+|z|^2)^s\left|\langle w,e^{iz\cdot x}\rangle\right|^2
\]
is finite (see, for example, Haroske and Triebel \cite[Remark 4.26]{HaTr08}). Then we compute, 
\begin{align*}
&\sum_{z\in\mathbb{Z}^{m'+m''}}(1+|z|^2)^s\left|\langle v,e^{iz\cdot x}\rangle\right|^2\\
&\qquad=\sum_{m\in\mathbb{Z}^{m'+m''}}(1+|z|^2)^s\left| \int_{\mathbb{T}^{m'+m''}} v(x) e^{-i z\cdot x} dx \right|^2\\
&\qquad=\sum_{(z',z'')\in\mathbb{Z}^{m'+m''}}(1+|z|^2)^s\left| \int_{\mathbb{T}^{m'}} \tilde{v}(x') e^{-i z'\cdot x'} dx \int_{\mathbb{T}^{n''}}e^{-i z''\cdot x''}dx'' \right|^2\\
&\qquad=(2\pi)^{2n''}\sum_{z'\in\mathbb{Z}^{n'}}(1+|z'|^2)^s\left| \int_{\mathbb{T}^{n'}} \tilde{v}(x') e^{-i m'\cdot x'} dx' \right|^2\\
&\qquad=(2\pi)^{2n''}\sum_{z'\in\mathbb{Z}^{m'}}(1+|z'|^2)^s\left|\langle \tilde{v},e^{iz'\cdot x'}\rangle\right|^2
\end{align*}
and the lemma follows.
\end{proof}

\begin{proof}[Proof of Theorem \ref{inOmega}] Many of the computations in this proof are very similar to those we have carried out to prove Proposition \ref{notCbeta>0} and \ref{notCbeta>alpha}. For example, to show that $u$ is  harmonic on $\Omega$ and continuous on $\overline\Omega$ it suffices to check that the series $\sum_{k=0}^\infty a_k\Re(x_1+ i x_2)^k $ converges absolutely and uniformly for $x_1^2+x_2^2\leq 1$, and thus for $x=(x_1,x_2,x'')$ in $\overline\Omega$. To do that we can resort to \eqref{notCbeta>0.eq1} and \eqref{notCbeta>alpha.eq1}. Also, arguing as in the proofs of Propositions \ref{notCbeta>0} and \ref{notCbeta>alpha} we can see that   
$u\in C^{0,\alpha}(\overline\Omega)$ and $u\notin C^{0,\alpha+\varepsilon}(\overline\Omega)$ for any $\varepsilon>0$. 

To complete the proof we have to verify the last statement about the membership (or, better, lack of membership) in $H^{\sigma+1/2}(\Omega)$. We first consider the case where $\alpha=0$.  For the sake of contradiction we assume that $u$ belongs to $H^{\sigma+1/2}(\Omega)$ for some $\sigma>0$. Then, by by the trace Theorem \ref{trace} we have that $u_{|\partial\Omega}\in H^{\sigma}(\partial\Omega)$ and in particular that $u_{\vert_{\mathbb{S}_1\times [-\delta,\delta]^{n-2}}}\in H^{\sigma}(\mathbb{S}_1\times [-\delta,\delta]^{n-2})$ for some $\delta>0$ as in \eqref{delta}. Since $\mathbb{T}\ni t\mapsto (\cos t, \sin t)\in\mathbb{S}_1$ is a $C^\infty$-diffeomorphism and $u_{\vert_{\mathbb{S}_1\times [-\delta,\delta]^{n-2}}}$ is constant---and therefore periodic---in the direction $x''$, we deduce that  $v(t,x''):=u(\cos t,\sin t,x'')$ belongs to $H^{\sigma}(\mathbb{T}^{n-1})$. Moreover, $v(t,x'')$ is constant in the direction of $x''$ and we have 
\[
v(t,x'')=\sum_{k=0}^\infty a_k \cos(kt)\quad\text{for all }(t,x'')\in\mathbb{T}^{n-1}\,.
\]
Hence, Lemma \ref{torustrace} implies that $\tilde{v}(t):=\sum_{k=0}^\infty a_k \cos(kt)$ belongs to $H^{\sigma}(\mathbb{T})$ and thus (again because $\mathbb{T}\ni t\mapsto (\cos t, \sin t)\in\mathbb{S}_1$ is a $C^\infty$-diffeomorphism) we have that the function 
\[
\tilde{u}(x):=\sum_{k=0}^\infty a_k\Re(x_1+ i x_2)^k\quad\text{of }x=(x_1,x_2)\in\mathbb{S}_1
\]
belongs to $H^{\sigma}(\mathbb{S}_1)$. However, $\tilde u$ is the trace on $\mathbb{S}_1$ of the function Proposition \ref{notCbeta>0} in the case where $n=2$ and, in the proof of Proposition \ref{notCbeta>0} we have seen that $\tilde u$ does not belong to $H^{\sigma}(\mathbb{S}_1)$ for any $\sigma>0$. So we have a contradiction and we deduce that $u$ does not belong to $H^{\sigma+1/2}(\Omega)$ for any $\sigma>0$. The proof for $0<\alpha<1$ is similar, we just use  Proposition \ref{notCbeta>alpha} instead of Proposition \ref{notCbeta>0}.
\end{proof}

We may now ``transplant'' the function $u$ from the set $\Omega$ to a ball $\mathbb{B}_n$. To do so we further assume that the function $\mu$ in the definition of $\Omega$ is decreasing (as, for example, the function in \eqref{supermu}). Then the corresponding set $\Omega$ is  star-convex  with respect to the origin. In particular, 
the map
\begin{equation}\label{zeta}
\zeta:\partial\Omega\ni x\mapsto \frac{x}{|x|}\in\mathbb{S}_{n-1}
\end{equation}
is a $C^\infty$-diffeomorphism. We have the following

\begin{thm}\label{fromOmega} Let $n\ge 3$, let $\Omega$ be as above with $\mu$ decreasing, let $\zeta$ be as in \eqref{zeta}, let $0\le \alpha<1$ and let $u$ be as in Theorem \ref{inOmega}. There exists a unique function $w\in C^{0,\alpha}(\overline{\mathbb{B}_n})$ that is harmonic in $\mathbb{B}_n$ and such that 
\[
w(\xi)=u\circ\zeta^{-1}(\xi)\quad\text{for all }\xi\in\mathbb{S}_{n-1}\,.
\]
Moreover, $w$ does not belong to $C^{0,\alpha+\epsilon}(\overline{\mathbb{B}_n})$ for any $\epsilon>0$ and does not belong to $H^{\sigma+1/2}(\mathbb{B}_n)$ for any $\sigma>0$ if $\alpha=0$ and for any $\sigma\ge \alpha$ if $0<\alpha<1$.
\end{thm}
\begin{proof} In the proof of Theorem \ref{inOmega} we have seen that the restriction of $u$ to the boundary of $\partial\Omega$ is in $C^{0,\alpha}(\partial\Omega)$ but not in $C^{0,\alpha+\epsilon}(\partial\Omega)$ for $\epsilon>0$ and not in $H^\sigma(\partial\Omega)$ for $\sigma>0$ if $\alpha=0$ and for $\sigma\ge \alpha$ if $0<\alpha<1$. Since $\zeta$ is a $C^\infty$-diffeomorphism it follows that $u\circ\zeta^{-1}$ belongs to $C^{0,\alpha}(\mathbb{S}_{n-1})$ but is not in $C^{0,\alpha+\epsilon}(\mathbb{S}_{n-1})$ and neither in $H^\sigma(\mathbb{S}_{n-1})$, with $\epsilon$ and  $\sigma$ as above. Then the existence of $w$ is a consequence of Proposition \ref{dirichlet}. We can see that $w\notin C^{0,\alpha+\epsilon}(\overline{\mathbb{B}_n})$ and the last statement about the (lack of) membership in $H^{\sigma+1/2}(\mathbb{B}_n)$ is a consequence of the trace  Theorem \ref{trace}. 
\end{proof}

The function $w$ meets both conditions \eqref{p1} and \eqref{p2} and it is explicit up to the solution of a Dirichlet boundary value problem. For example, we might obtain $w$ as a series of double layer potentials
\[
w=\sum_{k=0}^\infty a_k\, w^+[\psi_k]
\]
with densities  $\psi_k\in C^\infty(\mathbb{S}_{n-1})$ solutions of the boundary integral equations
\[
\frac{1}{2}\psi_k+W[\psi_k]=\Re(\zeta^{-1}_1+i\zeta^{-1}_2)^k\quad\text{on }\mathbb{S}_{n-1}\,.
\] 
Here we understand that $\zeta^{-1}_1$ and $\zeta^{-1}_2$ are the first and second component of $\zeta^{-1}$, respectively. Exploiting the symmetries of the support of integration $\mathbb{S}_{n-1}$ we can see that the equations above can also be written as
\[
\psi_k(x)+\fint_{\mathbb{S}_{n-1}}\frac{\psi_k(y)}{|x-y|^{n-2}}d\sigma_y=2\Re(\zeta^{-1}_1(x)+i\zeta^{-1}_2(x))^k
\]
for all $x\in\mathbb{S}_{n-1}$.

\section*{Acknowledgements}

The authors wish to thank  Pier Domenico Lamberti for the inspiring discussion that kindled the work on this topic and Valentina Casarino for the insights on spherical harmonics theory. Our gratitude also goes  to Martin Costabel, whose advice was essential for the last example. Some of the results here presented are also part of the  master thesis of Brian Luczak at the University of Tulsa \cite{Lu19}. 

Roberto Bramati is supported by the FWO Odysseus 1 grant G.0H94.18N: Analysis and Partial Differential Equations and by the Methusalem programme of the Ghent University Special Research Fund (BOF) (Grant number 01M01021).

Roberto Bramati and Matteo Dalla Riva are members of the ``Gruppo Nazionale per l'Analisi Matematica, la Probabilit\`a e le loro Applicazioni'' (GNAMPA) of the ``Istituto Nazionale di Alta Matematica'' (INdAM).

\end{document}